\documentclass{amsart}
\numberwithin{equation}{section}

\newtheorem{thm}{Theorem}[section]
\newtheorem{lem}[thm]{Lemma}

\newtheorem{defin}[thm]{Definition}

\newtheorem{remark}[thm]{Remark}

\begin{document}

\begin{center}
\textbf{{\large {\ Forward and inverse problems for a mixed-type equation with the Caputo fractional derivative and Dezin-type non-local condition }}}\\[0pt]
\medskip \textbf{R.R.Ashurov $^{1,2}$, U.Kh.Dusanova$^{1,3}$, N.Sh.Nuraliyeva$^{1,3}$}\\[0pt]
\textit{ashurovr@gmail.com, umidakhon8996@gmail.com, n.navbahor2197@gmail.com \\[0pt]}
\medskip \textit{\ $^1$V.I. Romanovskiy Institute of Mathematics,
  Uzbekistan Academy of Sciences
Tashkent, Uzbekistan;
\\
$^2$ Central Asian University,
264, Mirzo Ulugbek District, Milliy Bog St., 111221, Tashkent, Uzbekistan;
\\
$^3$Karshi State University,
Karshi 180119, Uzbekistan;}

\end{center}

\textbf{Abstract}: This work is dedicated to the study of a mixed-type partial differential equation involving a Caputo fractional derivative in the time domain $t > 0$ and a classical parabolic equation in the domain $t < 0$, along with Dezin-type non-local boundary and gluing conditions. The forward and inverse problems are studied in detail. For the forward problem, the existence and uniqueness of solutions are established using the Fourier method, under appropriate assumptions on the initial data and the right-hand side. We also analyze the dependency of solvability on the parameter $\lambda$, from the Dezin-type condition. For the inverse problem, where the right-hand side is separable as $F(x,t) = f(x)g(t)$ (the unknown function is  $f(x)$), the existence and uniqueness of a solution are proven under a certain condition on the function $g(t)$ (a constant sign is sufficient).


\section{Introduction}

Numerous researchers have investigated boundary value problems for differential equations of mixed type. These problems first attracted attention through the work of S. Chaplygin, who applied mixed-type partial differential equations to model gas dynamics. Later, A. Bitsadze \cite{Bits} demonstrated the ill-posedness of the Dirichlet problem for the equation $u_{xx} + \operatorname{sgn}(y) u_{yy} = 0$. Subsequent studies on various boundary value problems for such equations were compiled in the monographs \cite{Smir}, \cite{Moys}.

In this paper, we study a mixed type equation, one part of which contains a fractional derivative. Let us move on to the exact formulation of the problem under consideration.

 Let $0 < \rho < 1$. The Caputo fractional derivative of order $\rho$ of a function $f$ is given by (see \cite{Pskhu})
 $$
D_t^{\rho}f(t)=\frac{1}{\Gamma(1-\rho)}\int\limits_0^t \frac{f^{'}(\tau)}{(t-\tau)^{1-\rho}}d\tau,\quad t>0,
$$ 
provided the right-hand side exists. Here $\Gamma(\cdot)$ denotes the well-known  gamma function. 

Let $\Omega$ be an arbitrary $N$ dimensional domain with a sufficiently smooth boundary $\partial\Omega$.
Consider the following mixed type equation:
\begin{equation}\label{prob1.1}
\left\{
\begin{aligned}
& D_t^\rho u-\Delta u=F(x,t),\quad x\in\Omega, \quad 0<t\le\beta, \\
& u_{t}+\Delta u=F(x,t),\quad x\in\Omega, \quad -\alpha <t<0 ,\\
\end{aligned}
\right.\end{equation}
where $F(x,t)$ is a continuous function and $\alpha>0$, $\beta>0$ are given real numbers and $\Delta$ is the Laplace operator. 

\textbf{Dezin problem.} Find a function  $u(x,t)$ satisfying equation (\ref{prob1.1}) and the boundary condition  
\begin{equation}\label{prob1.2}
u(x,t)|_{\partial\Omega}=0,\quad t\in[\alpha,\beta],
\end{equation}
and gluing condition 
\begin{equation}\label{prob1.3}
\lim_{t\to +0}u(x,t)=\lim_{t\to -0}u(x,t),\quad x\in\Omega,
\end{equation}
and also a non-local condition
\begin{equation}\label{prob1.4}
u(x,-\alpha)=\lambda u(x,0),\quad x\in\Omega,
\end{equation}
where  $\lambda=const$, $\lambda\ne 0$.

This problem is called the Dezin problem due to condition \eqref{prob1.4}. Note, that if $\lambda=0$ then we arrive at the backward problem for subdiffusion equation.

\begin{defin}\label{def1}
A function  $u(x,t)\in AC(\overline{\Omega}\times [0,\beta])$  with the properties 

\begin{enumerate}
        \item
	$ u(x,t)\in C(\overline{\Omega}\times[-\alpha,\beta])$,
	\item
	$\Delta u(x,t)\in C(\overline{\Omega}\times(-\alpha,0)\cup (0,\beta])$,
	\item $D_{t}^{\rho}u(x,t)\in C(\overline{\Omega}\times(0,\beta])$,
    \item $u_{t}(x,t)\in C(\overline{\Omega}\times(-\alpha,0))$.
\end{enumerate}
and satisfying conditions \eqref{prob1.1}-\eqref{prob1.4} is called the (classical) solution of the problem \eqref{prob1.1}-\eqref{prob1.4}.
	\end{defin}

In equation \eqref{prob1.1}, the derivatives of the function $u(x,t)$ are considered in the open domain. The condition of continuity for these derivatives in the closed domain $\overline{\Omega}$, as suggested by O.A. Ladyzhenskaya \cite{Lad}, is imposed to facilitate a straightforward proof of the solution's uniqueness. The requirement of absolute continuity of the solution at $t\geq 0$ is necessary to exclude singular functions from consideration, due to which the uniqueness of the solution is violated. Notably, the solution derived via the Fourier method inherently satisfies these continuity and absolute continuity  requirements.

\textbf{Inverse problem.} Let $F(x,t)=f(x)g(t)$, and let the function $g(t)$ be known. Find functions $f(x)$ and $u(x,t)$ satisfying equation \eqref{prob1.1} such that 
$$f(x)\in C(\overline{\Omega})$$
and the function $u(x,t)$ satisfies conditions of Definition \ref{def1}, 
and, also an additional condition 
\begin{equation}\label{um8}
u(x, t_{0})=\varphi_{0}(x), \quad x\in\Omega,  
\end{equation}
here $\varphi_{0}(x)$ is given sufficiently smooth function and $t_{0}$ is given point in $(0,\beta)$.

In 1963 A. A. Dezin \cite{Dizin}  (see the condition ($\Gamma_1$)) studied solvable extensions of mixed-type differential equations. He formulated a boundary value problem characterized by 
$2\pi$-periodicity and nonlocal conditions, where in the value of the unknown function within a rectangular domain is related to the value of its derivative on the boundary. This formulation involves the Lavrentiev-Bitsadze operator and reflects a significant development in the theory of mixed-type equations.

In works  \cite{Sabitov}-\cite{Gushchina},  for mixed-type differential equations nonlocal boundary value problems of Dezin's type  have been investigated.  Let us dwell in more detail on  these works.

In \cite{Sabitov}, the following  degenerating mixed type equation is considered:
\begin{equation}\label{sab1}
Lu \equiv K(t)u_{xx} + u_{tt} - bK(t)u = F(x, t),
\end{equation}
in the rectangular domain $D = \{(x, t): 0 < x < l,\,\, -\alpha < t < \beta \}$, where \( K(t) = (\operatorname{sgn} t)|t|^m \), and \( m, b, l > 0 \) are given real constants. The study addresses a inhomogeneous Dezin-type non-local boundary condition of the form
\begin{equation*}
u_{t}(x, -\alpha) - \lambda u(x, 0) = \psi(x).
\end{equation*}

In \cite{Nakhusheva}, a similar problem is examined under the assumptions $m = b = 0$, $\alpha = l$, $\psi(x) = 0$, and $F(x, t) = f(x, t)H(t)$ ($H(t)$ is the Heaviside function), with $\lambda \ge 0$. It is also shown that in the case $\lambda < 0$, the homogeneous problem admits a nontrivial solution.

In \cite{Gushchina}, equation \eqref{sab1} is investigated under the same conditions as in \cite{Sabitov}, except in the homogeneous case where $F(x, t) \equiv 0$. It should be emphasized that all of the abovementioned works focus on forward problems.

In the work \cite{Mex}, the forward and inverse problems for equation \eqref{prob1.1} were studied. In solving the forward problem, instead of the non-local condition \eqref{prob1.4}, the gluing condition $D^{\rho}_{t}u(x,+0)=u_{t}(x,-0)$ was used. The inverse problem of determining the unknown function $f(x)$ was investigated for the case when $g(t)\equiv 1$.

 In \cite{KarPeng}, the inverse problem  is also considered, where the equation involves for $t > 0$  a Caputo fractional derivative of order $\rho$, and for $t < 0$ the equation is  hyperbolic type. Furthermore, in \cite{AshMar}-\cite{AshurFay}, a similar inverse problem are studied for the subdiffusion equation.

In this paper, we consider the forward problem \eqref{prob1.1}-\eqref{prob1.4} and study the inverse problem \eqref{prob1.1}-\eqref{um8} of determining the right-hand side.

In Section 2, we construct the solution of the forward problem using the Fourier method and demonstrate the uniqueness of the obtained solution. Further, in Section 3, we study the dependence of the existence of the solution to the forward problem \eqref{prob1.1}-\eqref{prob1.4} on the parameter $\lambda$. Since the elliptic part of the considered equation is the Laplace operator, the conditions imposed on the function $F(x,t)$ that ensure the existence of a solution are easily derived based on lemmas proven by V.A.Il'in. Section 4 is devoted to the study of the inverse problem \eqref{prob1.1}-\eqref{um8} of determining the right-hand side of the equation, assuming that the function $g(t)$ does not change sign. Then under a certain condition on the parametr $\lambda$ and $t_0$, we prove the existence and uniqueness of a solution to the inverse problem \eqref{prob1.1}-\eqref{um8}. Further, we will show that if this condition is violated, then for the existence of a solution to the inverse problem \eqref{prob1.1}-\eqref{um8}, it is sufficient that the function from the over-determination condition is orthogonal to some eigenfunctions of the Laplace operator with the Dirichlet condition.

\section{Preliminaries}

Let us denote by $\{v_k\}$ the complete orthonormal eigenfunctions in $L_{2}(\Omega)$ and by $\lambda_k$ (the values of $\lambda_k$, a sequence of non-negative integers that do not decrease with increasing index $k$: $0<\lambda_{1}\le\lambda_{2}\le\lambda_{3}\le...$) the set of positive eigenvalues of the following spectral problem (see, e.g. \cite{Ilin}, p. 93)
\begin{equation}\label{eq.1.5}
\left\{\begin{array}{c} {-\Delta v(x)=\lambda v(x)},\, \, x\in\Omega, \\ {v(x)\left|{}_{\partial \Omega } =0. \right. } \end{array}\right.
\end{equation}

Let $\sigma$ be an arbitrary real number. In the space $L_{2} (\Omega )$, we introduce the operator $\hat{A}^{{}^{\sigma } } $, which operates according to the rule
\begin{equation*}
\hat{A}^{{}^{\sigma } } g(x)=\sum _{k=1}^{\infty }\lambda _{k}^{\sigma }  g_{k} v_{k} (x). 
\end{equation*}
Here $g_{k} =(g,v_{k})$ are the Fourier coefficients of an element $g\in L_{2}(\Omega).$

Obviously, this operator $\hat{A}^{\sigma } \, \, $ with the domain 
\begin{equation*}
D\left(\hat{A}^{\sigma } \right)=\left\{g\in L_{2} \left(\Omega \right):\, \sum _{k=1}^{\infty }\lambda _{k}^{2\sigma } \left|g_{k} \right|^{2} <\infty \right\}
\end{equation*}
is selfadjoint. For elements from $D\left(\hat{A}^{\sigma } \right)$ we introduce the norm
\begin{equation*}
\left\| g\right\| _{\sigma }^{2} =\sum _{k=1}^{\infty }\lambda _{k}^{2\sigma } \left|g_{k} \right|^{2} =\left\| \hat{A}^{\sigma } g\right\| ^{2} . 
\end{equation*}
If we denote by $A$  the operator in $L_{2}\left(\Omega \right)$ acting according to the rule $Ag(x)=-\Delta g(x)$ and with the domain of definition $D(A)=\left\{g\in C^{2} (\overline{\Omega }):g(x)=0,\, \, x\in \partial \Omega\right\},$ then the operator $\hat{A}=\hat{A}^{1} $ is a selfadjoint extension in $L_{2}\left(\Omega\right)$ of the operator $A,$ (see \cite{Ilin}, p. 139).

Our reasoning will largely rely on the methodology developed in the monograph \cite{Kras}.

The following lemma is very important for study of the formulated problem. 

\begin{lem}\label{lem2.1}(see \cite{Kras}, p. 453)
 Let $\sigma >\frac{N}{4} $. Then  the following estimate 
$$||\hat{A}^{-\sigma } g||_{C(\Omega )} \le C||g||_{L_{2} (\Omega )}$$
holds.
\end{lem}

In order to prove the existence of a solution to the forward and inverse problems, it is necessary to study the convergence of the following series:
\begin{equation}\label{1.27}
\sum _{k=1}^{\infty }\lambda _{k}^{\tau }  \left|h_{k} \right|^{2} ,\, \, \, \, \, \, \, \, \, \, \, \, \tau >\frac{N}{2},
\end{equation}
here $h_{k}$   are the Fourier coefficients of the function $h(x)\in L_{2}(\Omega).$  In the case of integer $\tau$, in the  paper by Il'in \cite{Ilin} we obtain the conditions for convergence of such series in terms of the belonging of function $h(x)$ to the classical Sobolev space. In order to formulate this condition, let us introduce the class $\hat{W}_{2}^{1} \left(\Omega \right)$  as a closure by the $W_{2}^{1} \left(\Omega \right)$ norm  of the set of functions from $C^{\infty}_{0}(\Omega)$ that vanish  on the boundary of the domain $\Omega$. Il'in's lemma states that if function $h(x)$  satisfies the following conditions (we can take $\tau =\frac{N}{2} +1$, if $N$  is even and $\tau =\frac{N+1}{2}$, if  $N$  is odd)
$$h(x)\in W_{2}^{\left[\frac{N}{2} \right]+1} \left(\Omega \right),\, \, \, \, h(x),\, \Delta h(x),...,\Delta ^{\left[\frac{N}{4} \right]} h(x)\in \hat{W}_{2}^{1} \left(\Omega \right),
$$
then the series \eqref{1.27} converges. Here $W_{i}^{j}\left(\Omega \right)$  are  Sobolev spaces. Similarly, if in \eqref{1.27} $\tau$ is replaced by $\tau+1$, then the convergence conditions are
\begin{equation}\label{1.28}
h(x)\in W_{2}^{\left[\frac{N}{2} \right]+2} \left(\Omega \right),\, \, \, \, h(x),\, \Delta h(x),...,\Delta ^{\left[\frac{N}{4} \right]} h(x)\in \hat{W}_{2}^{1} \left(\Omega \right).
\end{equation}

Next we recall some properties of Mittag-Leffler function.

Let $\mu$ be an arbitrary complex number. The function defined by the following equation \begin{equation}\label{ml}
E_{\rho, \mu}(z)= \sum\limits_{k=0}^\infty \frac{z^k}{\Gamma(\rho
k+\mu)}
\end{equation} 
is called Mittag-Leffler function with two-parameters,  (see \cite{Pskhu}, p. 12). If the parameter $\mu =1$, then we have the classical Mittag-Leffler function: $ E_{\rho}(z)= E_{\rho, 1}(z)$.

\begin{lem}\label{lem2.2}(see \cite{Gor}, formula (4.4.5), p. 61)
 For any $t\geq 0$ one has
\begin{equation}\label{yan09}
0<E_{\rho, \mu}(-t)\leq \frac{C_0}{1+t},
\end{equation}
where constant $C_0$ does not depend on $t$ and $\mu$.
\end{lem}

\begin{lem}\label{lem2.3}(see~\cite{Gor}, p. 47)
 The classical Mittag--Leffler function of the negative argument $E_\rho(-t)$ is monotonically
decreasing function for all $0 <\rho < 1$ and
\begin{equation}\label{yan9}
 0<E_{\rho} (-t)<1,\quad E_{\rho} (0)=1.   
\end{equation}
\end{lem}

\begin{lem}\label{lem2.4}(see \cite{Gor}, formula (4.4.5), p. 61)
 Let $\rho > 0 $, $\mu>0$ and $\lambda \in C$. Then for all positive $t$ one has
\begin{equation}\label{10}
\int\limits_0^t \eta^{\rho-1}E_{\rho,\rho}(\lambda\eta^\rho)d\eta=t^{\rho} E_{\rho,\rho+1}(\lambda t^\rho).
\end{equation}
\end{lem}

\begin{lem}\label{lem2.5}(see \cite{Gor}, formula (4.2.3), p. 57)
 For all $\alpha>0$, $\mu\in \mathbb{C}$ the following recurrence relation holds: 
 $$E_{\rho, \mu}(-t)=\frac{1}{\Gamma(\mu)}-tE_{\rho, \mu+\rho}(-t).$$
\end{lem}

\begin{lem}\label{lem2.6}(see \cite{AshFay})
 Let $\lambda > 0$, $0 < \varepsilon < \rho$. Then, for all $t > 0$, the following coarser estimate holds:
 $$\left|t^{\rho -1} E_{\rho ,\rho } \left(-\lambda t^{\rho } \right)\right|\le  C\lambda ^{\varepsilon -1} t^{\varepsilon \rho -1},$$
where $C > 0$ is a constant independent of $\lambda$ and $t$.
\end{lem}

\section{Constructing the solution of the forward problem  (\ref{prob1.1})-(\ref{prob1.4})}

We seek the unknown function $u(x,t)$, which is a solution to the problem \eqref{prob1.1}-\eqref{prob1.4}, in the form
\begin{equation}\label{eq6}
u(x,t) = \sum_{k=1}^{\infty} T_k(t) v_k(x).
\end{equation}
Substituting the series (\ref{eq6}) into equation \eqref{prob1.1}, we obtain the following differential equations:
\begin{equation}\label{eq.1}
 D_t^\rho T_{k}(t) +\lambda_{k}T_{k}(t) =F_{k}(t),\quad t>0,
\end{equation}
\begin{equation}\label{eq..1}
 T_{k}^{'}(t) -\lambda_{k}T_{k}(t) =F_{k}(t),\quad t<0.
\end{equation}
where $F_k(t)$ are the Fourier coefficients of the function $F(x,t)$ with respect to the system of eigenfunctions $\{v_k(x)\}$, defined as the inner product in $L_2(\Omega)$.

The solutions to the above differential equations, respectively, have the form (see \cite{Kilbas}, p. 221).
$$T_{k}(t)= \left\{
\begin{aligned}
& a_{k}E_{\rho,1}(-\lambda_{k}t^{\rho})+\int\limits_{0}^{t}s^{\rho-1}E_{\rho,\rho}\left(-\lambda_{k}s^{\rho}\right)F_{k}(t-s)\,ds,\quad t>0, \\
&  b_{k} e^{\lambda_k t}-\int\limits_{t}^0 F_{k}(s) e^{\lambda_k(t-s)} ds, \quad t<0,\\
\end{aligned}
\right. $$
where $a_{k}$, $b_{k}$ - are arbitrary constants.

In order to find the unknown constants $a_{k}$, $b_{k}$ we use the gluing condition (\ref{prob1.3}). 
\[
T_{k}(0+0)=\left. \Bigl(a_{k}E_{\rho,1}(-\lambda_{k}t^{\rho})+\int\limits_{0}^{t}s^{\rho-1}E_{\rho,\rho}(-\lambda_{k}s^{\rho})F_{k}(t-s)\,ds\Bigr)\right|_{t=0}=a_{k},
\]
\[
T_{k}(0-0)=\left. \Bigl(b_{k} e^{\lambda_k t}-\int\limits_{t}^0 F_{k}(s) e^{\lambda_k(t- s)} ds\Bigr)\right|_{t=0}=b_{k}.
\]

To find the  coefficients $a_{k}$ , we use the non-local condition  (\ref{prob1.4}):
\begin{equation}\nonumber
a_{k} e^{-\lambda_k\alpha }-\int\limits_{-\alpha}^0  F_k(s) e^{\lambda_k(-\alpha- s)} ds=\lambda a_{k}   
\end{equation}
\begin{equation}\label{17.1}
  a_{k}\left(e^{-\lambda_k\alpha}-\lambda\right)=F^{*}_k,  
\end{equation}
where $F^{*}_k=\int\limits_{-\alpha}^0  F_k(s) e^{\lambda_k(-\alpha- s)} ds$.
In the sequel, we will use the notation
\begin{equation}\label{con0}
   \delta_{k}= e^{-\lambda_k\alpha}-\lambda, \quad k\ge 1.
\end{equation}
Thanks to this notation we rewrite \eqref{17.1} as: $a_{k}\delta_{k}=F^{*}_k.$

If for some $k$, we have $\delta_k=0$, then equation \eqref{17.1} has a solution only if the free term is zero, i.e., $F^{*}_k=0$: 
\begin{equation}\label{um9}
\int\limits_{-\alpha}^0  F_k(s) e^{\lambda_k (-\alpha-s)} ds=0.  \end{equation}
In this case, the coefficients $a_{k}$ remain arbitrary, and problem \eqref{prob1.1}-\eqref{prob1.4} does not have a unique solution. 

Thus, if $\delta_k\ne 0$ for all $k$, then the unknown coefficients $a_{k}$ are uniquely determined, and problem \eqref{prob1.1}-\eqref{prob1.4} has an unique solution. 

Indeed let  $u\equiv u_1-u_2 $. We have the following  homogeneous problem for $u(x,t)$:
\begin{equation}\label{homog1}
\left\{
\begin{aligned}
&D_t^{\rho} u(x,t) - \Delta u(x,t) = 0, && 0<t<\beta,\;\; x \in \Omega, \\
&u_t(x,t) +\Delta u(x,t) = 0, && -\alpha<t < 0,\;\; x \in \Omega,\\
\end{aligned}
\right.
\end{equation}
and the conditions  \eqref{prob1.2}, \eqref{prob1.3} and \eqref{prob1.4}.

Assume that \( u(x,t) \) satisfies all the conditions of the homogeneous problem, and let \( v_k \) be an arbitrary eigenfunction of the spectral problem (\ref{eq.1.5}) corresponding to the eigenvalue \( \lambda_k \).

Consider the function
\begin{equation}\label{eq10}
T_k(t) = \int\limits_{\Omega} u(x,t) v_k(x) \, dx, \quad k = 1, 2, \dots
\end{equation}
Differentiating under the integral sign with respect to \( t \), which is allowed by the definition of the solution, and using equation (\ref{homog1}), we obtain
$$D_t^{\rho} T_k(t) = \int\limits_{\Omega} D_t^{\rho} u(x,t) v_k(x) \, dx = \int\limits_{\Omega} \Delta u(x,t) v_k(x) \, dx,\quad t > 0,$$
$$\frac{d T_k(t)}{dt} = \int\limits_{\Omega} \frac{\partial u(x,t)}{\partial t} v_k(x) \, dx = -\int\limits_{\Omega} \Delta u(x,t) v_k(x) \, dx,\quad t < 0.$$
Integrating by parts and using condition (\ref{prob1.2}), we get:
$$D_t^{\alpha} T_k(t) = -\lambda_k T_k(t), \quad t > 0, \qquad
\frac{d T_k(t)}{dt} = \lambda_k T_k(t), \quad t < 0.$$
The solutions to these equations are given by (see \cite{Pskhu}):
\begin{equation}\label{eq13}
T_k(t) = a_{k}E_{\rho,1}(-\lambda_{k}t^{\rho}),\quad t > 0,
\end{equation}
\begin{equation}\label{eq15}
T_k(t) = b_{k} e^{\lambda_k t}, \quad k = 1, 2, \dots,\quad t < 0.
\end{equation}
To determine the constants $a_{k}$, $b_{k}$ we use the gluing conditions (\ref{prob1.3})  which translate into:
$$T_k(+0) = T_k(-0), \quad T_k(-\alpha)=\lambda T_k(0).$$
Solving this system gives $a_{k} = b_{k}$, $b_{k}=\lambda_{k}a_{k}$. Apply the non-local condition \eqref{prob1.4} to get: $a_{k}\delta_k=0.$

Since $\delta_k \neq 0$ for all $k \in \mathbb{N} $, then $a_{k}=b_{k}=0$. Therefore, from equations (\ref{eq13}) and (\ref{eq15}), the right-hand sides must be identically zero, which implies that \( u(x,t) \) is orthogonal to the complete system \( \{v_k(x)\} \). As a result, we conclude that \( u(x,t) \equiv 0 \) in \( \overline{\Omega} \).

Thus, we arrive at the following criterion for the uniqueness of the solution to problem \eqref{prob1.1}-\eqref{prob1.4}:

\begin{thm}\label{th3.1}
If there is a solution to problem (\ref{prob1.1})-(\ref{prob1.4}), then this solution is unique if and only if the condition $\delta_k\ne 0$ is
satisfied for all  $k\in \mathbb{N}$.
\end{thm}

So we obtain a formal solution to problem (\ref{prob1.1})-(\ref{prob1.4}) which is represented in the form
\begin{equation}\label{yech1}
u(x,t)= \left\{
\begin{aligned}
&\sum\limits _{k=1}^{\infty }\left(\frac{F_{k}^*}{\delta_k}E_{\rho,1}(-\lambda_{k}t^{\rho})+\int\limits_{0}^{t}s^{\rho-1}E_{\rho,\rho}\left(-\lambda_{k}s^{\rho}\right)F_{k}(t-s)\,ds\right)v_{k}(x),\quad 0\le t\le\beta, \\
&  \sum\limits _{k=1}^{\infty }\left(\frac{F_{k}^*}{ \delta_k} e^{\lambda_k t}-\int\limits_{t}^0 F_{k}(s) e^{\lambda_k(t-s)} ds\right)v_{k}(x), \quad -\alpha\le t\le 0.\\
\end{aligned}
\right.   
\end{equation}

To show that these series satisfy the conditions of Definition \ref{def1}, we need to estimate the denominator $\delta_k$ from below.

\section{Lower estimates for the denominator of the solution to the forward problem  (\ref{prob1.1})-(\ref{prob1.4})}

In this section, we investigate the conditions under which $\delta_k$ may be equal to zero, and for those cases where $\delta_k \ne 0$, we derive lower bounds for $\delta_k$. By its definition, $\delta_k$ depends on the parameter $\lambda$.
$$\delta_k:=\delta_k(\lambda)=e^{-\lambda_k\alpha}-\lambda, \quad k\ge 1.$$

It is not hard to see the following lemmas are true:

\begin{lem}\label{lem4.1}
 Let $\lambda\notin [0,1)$. Then  there exists a constant $\delta_0>0$ such that, for all $k \in \mathbb{N}$, the following estimate holds:
 $$|\delta_k|>\delta_{0},\quad \delta_{0}=\left\{
\begin{aligned}
&|\lambda|+e^{-\lambda_{1}\alpha} && \lambda<0, \\
&\lambda-e^{-\lambda_{1}\alpha} && \lambda\ge 1.\\
\end{aligned}
\right.$$
\end{lem}

\begin{thm}\label{th5.1}
Let $\lambda\notin [0,1)$. Let function $F(x,t)$ satisfy the conditions (\ref{1.28}) for all $t$. Then there exists a unique solution of problem  (\ref{prob1.1})-(\ref{prob1.4}), which is determined by the series \eqref{yech1}.
\end{thm}

\begin{proof}\vspace{-0.5em}
Let  formally differentiate series (\ref{yech1}). As a result we have

\begin{equation}\label{yech1.1}
-\Delta u(x,t)= \left\{
\begin{aligned}
&\sum\limits _{k=1}^{\infty }\left(\frac{\lambda_{k}F_{k}^*}{\delta_k}E_{\rho,1}(-\lambda_{k}t^{\rho})+\lambda_{k}\int\limits_{0}^{t}s^{\rho-1}E_{\rho,\rho}\left(-\lambda_{k}s^{\rho}\right)F_{k}(t-s)\,ds\right)v_{k}(x),\quad t>0, \\
&  \sum\limits _{k=1}^{\infty }\left(\frac{\lambda_{k}F_{k}^*}{ \delta_k} e^{\lambda_k t}-\lambda_{k}\int\limits_{t}^0 F_{k}(s) e^{\lambda_k(t-s)} ds\right)v_{k}(x), \quad t<0.\\
\end{aligned}
\right.   
\end{equation}

Consider the case for $t>0$, and in the case $t<0$  the absolute convergence of solution (\ref{yech1.1}) is proved in a similar way. This series is the sum of two series. We denote the first sum by $-\Delta S_{1}(x,t)$ , and the second by $-\Delta S_{2}(x,t)$. Let the partial sums of the first and second terms have the following forms, respectively:
\begin{equation}\label{eq38}
-\Delta S^{j}_{1} (x,t)=\sum\limits_{k=1}^{j}\frac{\lambda_{k}\left(\int\limits_{-\alpha}^0  F_k(s) e^{\lambda_k(-\alpha-s)} ds\right)E_{\rho,1}(-\lambda_{k}t^{\rho})}{\delta_k}v_{k}(x),
\end{equation}
\begin{equation}\label{eq39}
-\Delta S^{j}_{2} (x,t)=\sum\limits_{k=1}^{j}\lambda_k\left(\int\limits_{0}^{t}\eta^{\rho-1}E_{\rho,\rho}\left(-\lambda_{k}\eta^{\rho}\right)F_{k}(t-s)\,ds\right)v_{k}(x).
\end{equation}
In what follows, the symbol $C$ will denote a positive constant, not necessarily the same one. 

Let $\sigma > \frac{N}{4}$. Since  $\hat{A}^{-\sigma} v_{k} (x)=\lambda _{k}^{-\sigma }v_{k}(x),$ we have for (\ref{eq38}) 
\[
-\Delta S_{1}^{j} (x,t)=\hat{A}^{-\sigma } \sum\limits _{k=1}^{j}\frac{\lambda _{k}^{\sigma+1}\left(\int\limits_{-\alpha}^0  F_k(s) e^{\lambda_k(-\alpha-s)} ds\right)E_{\rho,1}(-\lambda_{k}t^{\rho})}{\delta_k}v_{k}(x).
\]
By virtue of Lemma \ref{lem2.1} we obtain 
\[
\begin{array}{l} {\left\| -\Delta S_{1}^{j} (x,t)\right\|^{2} _{C(\Omega )}\le C\left\|\sum \limits_{k=1}^{j}\frac{\lambda _{k}^{\sigma+1}\left(\int\limits_{-\alpha}^0  F_k(s) e^{\lambda_k(-\alpha-s)} ds\right)E_{\rho,1}(-\lambda_{k}t^{\rho})}{\delta_k}v_{k}(x)\right\|^{2} _{L_{2} (\Omega ) .}}   \end{array}
\]
Since the system $\left\{v_{k} \right\}$ is orthonormal, by applying Parseval's equality and using Lemma \ref{lem2.2} one has
\begin{equation}\nonumber
 \left\| -\Delta S_{1}^{j} (x,t)\right\|^{2} _{C(\Omega )}\le C t^{-2\rho} \sum\limits_{k=1}^{j}\lambda _{k}^{2\sigma}\left|\int\limits_{-\alpha}^0  F_k(s) e^{\lambda_k(-\alpha-s)} ds\right|^{2}.
\end{equation}
Applying the Cauchy-Schwarz  inequality
$$
\left\| -\Delta S_{1}^{j} (x,t)\right\|^{2} _{C(\Omega )}\le \frac{Ct^{-2\rho}}{\lambda^{2}_{1}}\int\limits_{-\alpha}^{0}\sum\limits_{k=1}^{j}\lambda _{k}^{2\sigma}\left|F_k(s) \right|^{2}\,ds,\quad \tau=2\sigma>\frac{N}{2}.$$
This means that we have series, similar to the series \eqref{1.27}. Thus, if the function $F(x,t)$ satisfies conditions \eqref{1.28} with $\tau>\frac{N}{2}$, then the series $\left| -\Delta S_{1}(x,t)\right| _{C(\overline\Omega )}^{2}\le C, $ $t> 0$, will converge.

For the series \eqref{eq39} by virtue of Lemma \ref{lem2.6}  we get
\begin{equation*}
\left\| -\Delta S_{2}^{j} (x,t)\right\| _{C(\Omega )}^{2} \le C\sum \limits_{k=1}^{j}\left|\int _{0}^{t}s^{\varepsilon\rho-1} \lambda _{k}^{\sigma +\varepsilon } F_{k} (t-s)ds\right|^{2}.
\end{equation*}
Further, will apply the generalized Minkowski inequality. Then 
\begin{equation}\label{eq445}
\left\| -\Delta S_{2}^{j}(x,t)\right\| _{C(\Omega )}^{2} \le C\left[\int _{0}^{t}s^{\rho\varepsilon -1} \left(\sum \limits_{k=1}^{j}\left|\lambda _{k}^{2(\sigma +\varepsilon) }|| F_{k} (t-s)\right|^{2}  \right)^{\frac{1}{2} } ds \right]^{2}, \quad 
\tau=2\sigma +2\varepsilon  >\frac{N}{2}.
\end{equation}

Here we again get a series similar to \eqref{1.27}. In this case, $\tau=2\sigma +2\varepsilon$. Since $\varepsilon$ is an arbitrarily small number, then the series (\ref{eq445}) converges under the same conditions (\ref{1.28}) for the function $F(x,t)$.

Consequently,  $\left| -\Delta S_{1}(x,t)\right| _{C(\overline\Omega )}^{2}\le C, $ $\left|-\Delta S_{2} (x,t)\right|_{C\left (\Omega \right)}^{2} \le C,\, \, \, \, \, t>0$. Thus $\Delta u(x,t)\in C(\overline{\Omega}\times(0,\beta))$, in particular $u(x,t)\in C(\overline{\Omega}\times[0,\beta])$. Using completely similar reasoning, it can be shown that sum (\ref{yech1.1}) at $t<0$ has the same properties as sum (\ref{yech1.1}) at $t>0$ . Hence, $\Delta u(x,t)\in C(\overline{\Omega}\times(-\alpha,0))$, in particular $u(x,t)\in C(\overline{\Omega}\times[-\alpha,0])$.

From equation \eqref{prob1.1},  we have  $D_{t}^{\rho}u(x,t)\in C(\overline{\Omega}\times(0,\beta)), \, \, \,  u_{t}(x,t)\in C(\overline{\Omega}\times(-\alpha,0)).$ That $u(x, t)$ is absolutely continuous in a closed region follows from the fact that every function $T_k(t) v_k(x)$ is such. Theorem \ref{th5.1} is proved.
\end{proof}

\begin{lem}\label{lem4.2}
 Let $0<\lambda <1$. Then there exists a number $k_0 \in \mathbb{N}$, such that for all $k > k_0$, the following estimate holds:
 $$\left| \delta_k \right| \ge \frac{\lambda}{2}.$$
\end{lem}
 
If $0<\lambda<1$, then obviously, there is a unique $\lambda_0>0$ such that $e^{-\lambda_{0}\alpha}=\lambda$. If $\lambda_k\neq \lambda_0$ for all $k\in\mathbb{N}$ then the formal solution of problem \eqref{prob1.1}-\eqref{prob1.4} has the form \eqref{yech1}. 

If $\lambda_k=\lambda_0$ for $k = k_{0}, k_{0} + 1, \ldots, k_{0} + p_{0} - 1$, where $p_{0}$ is the multiplicity of the eigenvalue $\lambda_{k_0}$, then for the solvability of problem \eqref{prob1.1}-\eqref{prob1.4}  it is necessary and sufficient that the following equality holds (see \eqref{17.1}):
    \begin{equation}\label{orto}
    F^{*}_k = (F^{*}, v_k) = 0, \quad k \in K_0, \quad K_0 = \{k_0, k_0 + 1, \ldots, k_0 + p_0 - 1\}.
    \end{equation} 
 In this case, the solution of problem \eqref{prob1.1}-\eqref{prob1.4} can be written as follows:
\begin{equation}\label{um10}
u(x,t) = 
\begin{cases}
\displaystyle 
\sum\limits_{k \notin K_0} 
\left(
    \frac{F_{k}^*}{\delta_k} E_{\rho,1}(-\lambda_{k} t^{\rho}) 
    + \int\limits_{0}^{t} s^{\rho-1} E_{\rho,\rho}(-\lambda_{k} s^{\rho}) F_k(t-s)\, ds
\right) v_k(x) 
\\[0.5em]
\displaystyle \quad + \sum\limits_{k \in K_0} a_k E_{\rho,1}(\lambda_k t^{\rho}) v_k(x), 
\quad t > 0, \\[1em]

\displaystyle 
\sum\limits_{k \notin K_0} 
\left(
    \frac{F_k^*}{\delta_k} e^{\lambda_k t} 
    - \int\limits_{t}^{0} F_k(s) e^{\lambda_k(t-s)} ds
\right) v_k(x) 
+ \sum\limits_{k \in K_0} a_k e^{\lambda_k t} v_k(x), 
\quad t < 0,
\end{cases}
\end{equation}
\noindent
here, $a_k$ are arbitrary constants.

Thus, we obtain the following statement:

\begin{thm}\label{th5.2}
Let $0<\lambda<1$ and  let function $F(x,t)$ satisfy the conditions (\ref{1.28}) for all $t$.

1) If  $\lambda_k\neq \lambda_{0}$, for all $k\ge 1$, then there exists a unique solution of the problem  (\ref{prob1.1})-(\ref{prob1.4}) and it can be represented as the form \eqref{yech1}.

2) If $\lambda_{k} = \lambda_{0}$, for some $k$ and the orthogonality condition \eqref{orto} holds for indices $ k \in K_0$, then the problem  (\ref{prob1.1})-(\ref{prob1.4}) has a solution, which is expressed in the form (\ref{um10}) with arbitrary coefficients $a_{k}$.
\end{thm}

\begin{proof}\vspace{-0.5em}   
We have considered the proof of the first part of the theorem above in Theorem \ref{th5.1}. Now, we need to show the convergence of the series \eqref{um10}. 
If $k\in K_{0}$, then in the solution (\ref{um10}) are formed additional series as: 
$$
u_{0}(x,t)=\left\{
\begin{aligned}
& \sum\limits_{k\in K_{0}}a_{k}E_{\rho ,1}\left(-\lambda_{k}t^{\rho}\right)v_{k} (x), \quad t>0, \\
& \sum\limits_{k \in K_0} a_k e^{\lambda_k t} v_k(x), \quad t<0.\\
\end{aligned}
\right.$$
Since $K_{0}$ has a finite number of elements, than these series, consist of finite sum of smooth functions. Therefore, these series satisfy all conditions of Definition \ref{def1}.
\end{proof}

\section{ Existence and uniqueness of the solution of the inverse problem (\ref{prob1.1})-(\ref{um8})}

We study the inverse problem for equation \eqref{prob1.1} with the right-hand side of the form $F(x,t) = f(x)g(t)$, where $g(t)$ is a given function and $f(x)$ is an unknown function. Furthermore, since we use the solution of the forward problem when solving the inverse problem, in all subsequent sections we assume that $\delta_k \neq 0$ for all $k$. According to the additional condition \eqref{um8}, it is sufficient to construct the solution of the inverse problem \eqref{prob1.1}-\eqref{um8} only for $t > 0$.  Using representation \eqref{yech1}, we obtain the following solution to the inverse problem \eqref{prob1.1}-\eqref{um8}:
\begin{equation}\label{yech2}
 u(x,t)= \sum\limits _{k=1}^{\infty }\left(a_{1k}E_{\rho,1}(-\lambda_{k}t^{\rho})+f_{k}\int\limits_{0}^{t}s^{\rho-1}E_{\rho,\rho}\left(-\lambda_{k}s^{\rho}\right)g(t-s)\,ds\right)v_{k}(x),\quad 0\le t\le\beta,  
\end{equation}
where 
$$a_{1k}=\frac{f_{k}\int\limits_{-\alpha}^0  g(s) e^{\lambda_k(-\alpha- s)} ds}{\delta_k}.$$
Substituting the function \eqref{yech2} into the condition \eqref{um8}, we obtain the equation 
\begin{equation}\label{N14}
 \sum\limits_{k=1}^\infty  T_{k}(t_{0})v_{k}(x)=\varphi_{0}(x)=\sum\limits_{k=1}^\infty \varphi_{0k}v_{k}(x), 
\end{equation}
where 
$$T_{k}(t_0)=\frac{f_{k}\int\limits_{-\alpha}^0  g(s) e^{\lambda_k(-\alpha- s)} ds}{\delta_k}E_{\rho,1}(-\lambda_{k}t_0^{\rho})+f_{k}\int\limits_{0}^{t_0}\eta^{\rho-1}E_{\rho,\rho}\left(-\lambda_{k}\eta^{\rho}\right)g(s)\,ds,$$
and 
\begin{equation}\nonumber
\varphi_{0k} = \int\limits_{\Omega} \varphi_{0}(x) v_k(x) \, dx, \quad k = 1, 2, \dots ,
\end{equation}
the numbers $f_{k}$ are so far unknown and have to be determined.

From relation \eqref{N14}, we have
\begin{equation}\label{N16}
f_{k}\Delta_{k}(t_{0})=\delta_k\varphi_{0k}=(e^{-\lambda_k\alpha}-\lambda)\varphi_{0k},    
\end{equation}
here
\begin{equation}\nonumber
  \Delta_{k}(t_{0})=E_{\rho,1}(-\lambda_{k}t_0^{\rho})\int\limits_{-\alpha}^0  g(s) e^{\lambda_k(-\alpha- s)} ds +(e^{-\lambda_k\alpha}-\lambda)\int\limits_{0}^{t_0}s^{\rho-1}E_{\rho,\rho}\left(-\lambda_{k}s^{\rho}\right)g(t_{0}-s)\,ds.
\end{equation}
Let us introduce the following notation:
$$I_{k}(\alpha)=\int\limits_{-\alpha}^0  g(s) e^{\lambda_k(-\alpha- s)} ds,\quad I_{k,\rho}(t_0)=\int\limits_{0}^{t_0}s^{\rho-1}E_{\rho,\rho}\left(-\lambda_{k}s^{\rho}\right)g(t_{0}-s)\,ds.$$
Again, as we noted above, if $\Delta_{k}(t_0)\neq 0$  for all $k$, then the coefficients $f_{k}$  are found uniquely, otherwise, i.e. if $\Delta_{k}(t_0)=0$ for some $k$, according to the equation (\ref{N16}), we can see that the coefficients $f_{k}$ are chosen arbitrarily. Therefore, we have the following uniqueness criterion for the inverse problem \eqref{prob1.1}-\eqref{um8}:

\begin{thm}\label{th6.1}
The uniqueness of the solution to the inverse problem \eqref{prob1.1}-\eqref{um8} is guaranteed if and only if $\Delta_k(t_0)\ne 0$ for all $k\ge 1$.
\end{thm}

\begin{proof}\vspace{-0.5em}
It is sufficient to prove that if $\Delta_k(t_0)\neq 0$ for all $k \ge 1$, then the solution is unique. Assume the contrary, there are two different solutions $\{u_1,f_1\}$ and $\{u_2,f_2\}$ satisfying the inverse problem \eqref{prob1.1}-\eqref{um8}. We need to show that $u\equiv u_1-u_2 \equiv 0$, $f\equiv f_1-f_2\equiv 0$. For $\{u,f\}$ we have the following problem:
$$\left\{
\begin{aligned}
& D_t^\rho u-\Delta u=f(x)g(t),\quad 0<t<\beta, \\
& u_{t}+\Delta u=f(x)g(t), \quad -\alpha <t<0,\\
& u(x,t)|_{\partial\Omega}=0,\quad -\alpha< t< \beta, \\
& u(x,+0)=u(x,-0),\\
& u(x,-\alpha)=\lambda u(x,0) ,\quad x \in \Omega,\\
& u(x, t_{0})=0, \quad x \in \Omega, \quad 0<t_0<\beta.
\end{aligned}
\right.$$
We take any solution $\{u,f\}$ and define $T_k(t)=(u,v_k)$ and $f_k=(f,v_k)$. Then, due to the self-adjointness of the operator $\Delta$ and the continuity of the derivatives of the solution up to the boundary of the domain $\Omega$, we have 
$$
  D_t^\rho T_k(t)= (D_t^\rho u, v_k)=(\Delta u, v_k)+f_k g(t) =( u,\Delta v_k)+f_k g(t) =-\lambda_k T_k(t)+f_k g(t),  
$$
$$
  \frac{\partial T_k(t)}{\partial t}= \left(\frac{\partial u}{\partial t}, v_k\right)=(-\Delta u, v_k)+f_k g(t) =( u,-\Delta v_k)+f_k g(t) =\lambda_k T_k(t)+f_k g(t).  
$$
Therefore, for $T_k$ we obtain the  the following differential equations:
$$
\left\{
\begin{aligned}
& D_{t}^{\rho } T_{k} (t)+\lambda _{k} T_{k} (t)=f_{k}g(t), \quad t>0, \\
& T_{k}^{'} (t)-\lambda _{k} T_{k} (t)=f_{k}g(t), \quad t<0,\\
\end{aligned}
\right.
$$
where $f_{k}$ is the Fourier coefficients of the function $f(x)$ according to the system of eigenfunctions $\left\{v_{k} (x)\right\} $ defined as the scalar product in $L_{2} (\Omega).$

The solutions of the last differential equations, respectively have the form:
$$
T_{k}(t)= \left\{
\begin{aligned}
& A_{k}E_{\rho,1}(-\lambda_{k}t^{\rho})+f_{k}\int\limits_{0}^{t}s^{\rho-1}E_{\rho,\rho}\left(-\lambda_{k}s^{\rho}\right)g(s)\,ds,\quad t>0, \\
& B_{k} e^{\lambda_k t}-f_{k}\int\limits_{t}^0 g(s) e^{\lambda_k(t-s)} ds,\quad t<0.\\
\end{aligned}
\right.   
$$
The resulting differential equation satisfies the conditions (\ref{prob1.3}), (\ref{prob1.4}), and (\ref{um8}), which transform into the following:
$$T_k(-0) = T_k(+0), \quad T_k(-\alpha) = \lambda T_k(0), \quad T_k(t_0) = 0, \quad k\ge 1.$$
Now, based on these conditions, we will prove the uniqueness of the solution to the inverse problem.

From the condition $ T_k(-0) = T_k(+0)$, and according to the above calculations, it follows that $ A_k = B_k$. 

Taking into account the condition $T_k(-\alpha) = \lambda T_k(0)$, we obtain the following:
\begin{equation}\label{N19}
 A_{k}=\frac{f_{k}\int\limits_{-\alpha}^0  g(s) e^{\lambda_k(-\alpha- s)} ds}{\delta_k}.   
\end{equation}

Now, let us verify the final condition $T_k(t_0) = 0$. This leads to the following equation:
\begin{equation}\label{N18}
f_k\left( E_{\rho,1}(-\lambda_k t_0^{\rho}) \int\limits_{-\alpha}^{0} g(s) e^{\lambda_k(-\alpha - s)}\, ds + (e^{-\lambda_k \alpha} - \lambda) \int\limits_{0}^{t_0} s^{\rho - 1} E_{\rho,\rho}(-\lambda_k s^{\rho}) g(t_0 - s)\, ds \right) = 0.
\end{equation}
Since $ \Delta_k(t_0) \ne 0$ for all $k \in \mathbb{N} $, the equality \eqref{N18} holds if and only if $f_k = 0$. From this, it follows by equation \eqref{N19} that $A_k = B_k = 0$. Then due to completeness of the set of eigenfunctions $\{v_k\}$ in $L_2(\Omega)$, we finally have  $f(x)\equiv 0$ and  $u(x,t)\equiv0$.
\end{proof}

Theorem \ref{th6.1} is proved.

\section{Lower estimates for the denominator of the solution to the inverse problem 
(\ref{prob1.1})-(\ref{um8})}

We now provide a lower estimate for $\Delta_{k}(t_{0})$. To do this, it is required that the function $g(t)$ does not vanish. Since  $g(t)\neq 0$, and it is continuous, the analysis remains the same whether $g(t)>0$ or $g(t)<0$. Therefore, in the four lemmas to be proved below, we will assume without loss of generality that $g(t)>0$.

Let $g \in C[-\alpha, \beta]$ and $g(t) \neq 0$, we define
\begin{equation}\label{minmax}
 m= \min_{t \in [-\alpha,t_0]} \{g(t)\} > 0, \quad M= \max_{t \in [-\alpha,t_0]} \{g(t)\}>0.  
\end{equation}

\begin{lem}\label{lem7.1}
Let $\lambda<0$,  $g(t)\in C[-\alpha,\beta]$ and $g(t)\neq 0$, $t\in [-\alpha,\beta]$. Then, there is a constant $C >0$, depending on $t_0$ and $\alpha$, such that for all $k$:
$$\Delta_k (t_0)\geq \frac{C}{\lambda_k}.$$
\end{lem}

\begin{proof}\vspace{-0.5em}
For $t_{0}\in (0, \beta]$, we have (see Lemma \ref{lem2.4})
$$I_{k,\rho}(t_0) \geq m\int\limits _0^{t_0} s^{\rho-1} E_{\rho, \rho} (-\lambda_k  s^\rho)ds =m t_{0} ^\rho E_{\rho, \rho+1} (-\lambda_k t_{0}^\rho ).$$
Taking into account (see Lemma \ref{lem2.5}) 
$$E_{\rho, \rho+1}(-t)=t^{-1} (1- E_{\rho} (-t)),$$
we obtain
$$I_{k,\rho}(t_{0})\geq \frac{1}{\lambda_k} (1- E_{\rho} (-\lambda_k t_{0}^\rho))m\geq \frac{1}{\lambda_k} (1- E_{\rho} (-\lambda_1 t_{0}^\rho))m\geq  \frac{C_{t_0}}{\lambda_k}, \,\, C_{t_{0}}>0,$$
$$I_{k}(\alpha)\geq m\int\limits_{-\alpha}^0 e^{\lambda_k(-\alpha- s)} ds=m\frac{1-e^{-\lambda_k\alpha}}{\lambda_{k}} \geq  \frac{C_{\alpha}}{\lambda_k}, \,\, C_{\alpha}>0.$$   
Therefore
$$\Delta_k (t_0)\geq {E_{\rho,1}(-\lambda_k  t_0^\rho)\frac{C_{\alpha}}{\lambda_k}}+{(e^{-\lambda_k\alpha}-\lambda)\frac{C_{t_0}}{\lambda_k}}\geq {(e^{-\lambda_k\alpha}-\lambda)\frac{C_{t_0}}{\lambda_k}},$$
which implies the desired assertion due to  $\lambda<0$. Lemma \ref{lem7.1} is proved.
\end{proof}

\begin{lem}\label{lem7.2}
Let  $\lambda\ge 1$, $g(t)\in C[-\alpha,\beta]$ and $g(t)\neq 0$, $t\in [-\alpha,\beta]$. 

If number $t_{0}$ satisfies following condition 
  \begin{equation}\label{N1}
   t^{\rho}_{0}>\frac{C_0}{\lambda_{1}}\left(1+\frac{M}{m}\right),   
  \end{equation}
where $C_0$ is the number in the Lemma \ref{lem2.2} then, there is  constant $C>0$ depending on $t_0$, $\rho$ and $\alpha$, such that for all  $k$:
\begin{equation}\label{est2}
  	|\Delta_k (t_0)|\ge \frac{C}{\lambda_k}.  
\end{equation}

If the number $t_{0}$ does not satisfy condition \eqref{N1}, then there exist  a number $k_l$, $l\in\mathbb{N}$  such that the  estimate \eqref{est2} holds for all $k > k_l$.
\end{lem}

\begin{proof}\vspace{-0.5em}
We begin by estimating $\Delta_k(t_0)$ from below. From its definition, it consists of a sum of two integrals. For the first integral, we have:
\begin{equation}\label{min}
 \int_{-\alpha}^0 g(s) e^{\lambda_k(-\alpha - s)}\, ds \geq m  \frac{1 - e^{-\lambda_k \alpha}}{\lambda_k}, \end{equation}
For the second integral, using Lemma \ref{lem2.4} and Lemma \ref{lem2.5} again, we get
\[
\int_0^{t_0} s^{\rho - 1} E_{\rho,\rho}(-\lambda_k s^\rho) g(t_0 - s)\, ds \leq M  \frac{1 - E_{\rho,1}(-\lambda_k t_0^\rho)}{\lambda_k}.
\]
Hence,
\[
\Delta_k(t_0) \geq \frac{E_{\rho,1}(-\lambda_k t_0^\rho)}{\lambda_k} \left[ m (1 - e^{-\lambda_k \alpha}) + (\lambda - e^{-\lambda_k \alpha}) M \right] - \frac{M}{\lambda_k},
\]
which implies
$$\Delta_k(t_0) \geq -\frac{M}{\lambda_k},$$
where $C_1 = M$.

Next, to estimate $\Delta_k(t_0)$ from above, we write:
\[
\Delta_k(t_0) \leq M  \frac{1 - e^{-\lambda_k \alpha}}{\lambda_k}  E_{\rho,1}(-\lambda_k t_0^\rho) - m  (\lambda - e^{-\lambda_k \alpha})  \frac{1 - E_{\rho,1}(-\lambda_k t_0^\rho)}{\lambda_k}.
\]
Using the Lemma \ref{lem2.2}  we obtain:
\begin{equation}\label{N2}
 \Delta_k(t_0) \leq \left( \frac{1 - e^{-\lambda_k \alpha}}{\lambda_k} \right) \left( \frac{C_{0}(M + m)}{\lambda_k t_0^\rho} - m \right).   
\end{equation}
Note that the expression in parentheses becomes negative under the assumption:
\[
t_0^\rho > \frac{C_0}{\lambda_1} \left(1 + \frac{M}{m} \right).
\]
Thus, for all $k \in \mathbb{N}$, we have:
\[
\Delta_k(t_0) \leq -\frac{C_2}{\lambda_k},
\]
where $C_{2} =\left(\frac{M + m}{\lambda_1 t_0^\rho} - m\right)>0$. 

Hence, there exists a constant $C=\min\{C_{1}, C_{2}\}$ such that the required lower bound holds. 

Now let $\lambda\ge 1$ and  assume that, condition (\ref{N1}) not be satisfied for the given values of the parametr. However, there exists an index $k_{l}$, such that, for all $k>k_l$ the condition  $t_0^{\rho} > \frac{C_0}{\lambda_k}\left(1 + \frac{M}{m}\right)$ is satisfied, since  $\frac{C_0}{\lambda_k}\left(\frac{M + m}{m}\right) \to 0$  as $ k \to \infty,$ (see \eqref{N2}). Therefore, for all $k>k_l$ the estimate \eqref{est2} holds. Lemma \ref{lem7.2} is proved.
\end{proof}

\begin{lem}\label{lem7.5}
 Let $0 < \lambda < 1$, $g(t)\in C[-\alpha,\beta]$ and $g(t)\neq 0$, $t\in [-\alpha,\beta]$. Then for all $k> k_r$, $r\in\mathbb{N}$ the following estimate  
 \begin{equation}\label{est53}
   | \Delta_{k,\rho}(t_0)|\ge \frac{C}{\lambda_k},  
 \end{equation}
 is valid, where constant $C>0$ depend on $\rho$,  $t_0$ and $\alpha$.
\end{lem}

\begin{proof}\vspace{-0.5em}

Since $\delta_k \neq 0$, it follows that $\lambda_k \neq \lambda_0$ for all $k$. Therefore, we consider only the following two cases.

Case 1. Let $\lambda_k < \lambda_0$.
In this case, based on the proof of Lemma \ref{lem7.1}, it is not difficult to see that for all $k < k_0$, the following estimate holds:
$$ \Delta_{k,\rho}(t_0)>c_{0},$$
where $c_0 > 0$ is a constant depending on $\alpha$, $t_0$, and $\rho$.

Case 2. Let $\lambda_{k}>\lambda_{0}$. We prove this  case of lemma  similarly to the proofs of the previous lemmas. The lower bound of $\Delta_k(t_0)$ has the form ( see Lemma \ref{lem7.2}) 
$$\Delta_k(t_0)\ge -\frac{C_{1}}{\lambda_{k}}.$$

Now, we establish an upper bound for $\Delta_k(t_0)$. To this end, using  Lemma \ref{lem2.4}, Lemma \ref{lem2.5} and Lemma \ref{lem2.2}, we obtain :
$$\Delta_k(t_0) \le \frac{C_0}{\lambda_k^2 t_0^\rho} \left( M(1 - e^{-\lambda_k \alpha}) + (\lambda - e^{-\lambda_k \alpha}) m \right) - \frac{(\lambda - e^{-\lambda_k \alpha}) m}{\lambda_k}.$$
Thus, for all $k>k_r$, we have
$$\Delta_k(t_0) \leq -\frac{C_3}{\lambda_k},$$
where $C_{3} =\left(\lambda-e^{-\lambda_{k}\alpha}\right)m >0$.

Therefore, there exists a constant $C=\min\{c_{0}, C_{1}, C_{3}\}$ such that, for all $k> k_r$ the required lower bound holds. The lemma \ref{lem7.5} is proved.
\end{proof}

The above estimates \eqref{est2} and \eqref{est53}  show that there exists an index $k_1$ such that $\Delta_k(t_0)\ne 0$ for all $k > k_1=\max\{k_l, k_r\}$. However, in some cases, it may happen that $\Delta_{k}(t_0)=0$ for $k<k_{1}$. Naturally, the question may arise as to how the indices $k_l$ and $k_r$ can be determined.

For example, according to the proof of the second condition of Lemma \ref{lem7.2}, the index $k_l$ is given by
$$k_l = \min\left\{k : t_0^{\rho} > \frac{1}{\lambda_k}\left(1 + \frac{M}{m}\right)\right\}.$$
According to the proof of the second condition of Lemma \ref{lem7.5}, the index $k_r$ is given by
$$k_r = \min\left\{k : \frac{(\lambda - e^{-\lambda_k \alpha}) m}{\lambda_k}>\frac{C_0}{\lambda_k^2 t_0^\rho} \left( M(1 - e^{-\lambda_k \alpha}) + (\lambda - e^{-\lambda_k \alpha}) m \right)\right\}.$$
Hence, we introduce the set
$$\mathbb{K}_0 = \{k \in \mathbb{N} : \Delta_k(t_{0}) = 0\}.$$
\begin{remark}\label{deltaDelta}
Note, that if $k\in \mathbb{K}_0$, then obviously $\delta_k \neq 0$.
\end{remark}

\begin{lem}\label{lem7.3}
The set $\mathbb{K}_0$ is either empty or contains only finitely many elements.
\end{lem}
\begin{proof}\vspace{-0.5em}
From the proof of Lemma \ref{lem7.2}, it follows that if there exists an index $k \in \mathbb{K}_0$, then necessarily $k\le k_l$. Therefore, $\mathbb{K}_0$ is a finite set.
Moreover, as mentioned in Section 1, the sequence $\{\lambda_k\}$ consists of discrete values. Hence, $\Delta_k(t_{0})$ can vanish only at isolated indices, and it is possible that no such index exists. In this case, the set $\mathbb{K}_0$ is empty. A similar argument is valid for the elements of the set $\mathbb{K}_0$ when $k \leq k_r$. This completes the proof of Lemma \ref{lem7.3}.
\end{proof}

\begin{thm}\label{th7.1}
Let  $g(t)\in C[-\alpha,\beta]$ and  $g(t)\neq 0$, $t\in [-\alpha,\beta]$. Let $\lambda<0$. Then there exists a unique solution of the inverse problem (\ref{prob1.1})-(\ref{um8}) and it can be represented as:
\begin{equation}\label{tyech1}
\begin{array}{l} {u(x,t)=\sum_{k=1}^{\infty }\left(\frac{\varphi_{0k}}{\Delta_{k}(t_0)}E_{\rho,1}(\lambda_{k}t^{\rho})\int\limits_{-\alpha}^0  g(s) e^{\lambda_k(-\alpha- s)} ds\right)v_{k}(x)+} \\\\ 

{+\sum _{k=1}^{\infty }\left( \frac{\varphi_{0k}(e^{-\lambda_{k}\alpha}-\lambda)}{\Delta_{k}(t_0)}\int\limits_{0}^{t}s^{\rho-1}E_{\rho,\rho}\left(-\lambda_{k}s^{\rho}\right)g(t-s)\,ds\right)v_{k}(x),\, \, \, t>0,\,} \end{array}
\end{equation}
$$u(x,t)=\sum _{k=1}^{\infty }\left(\frac{\varphi_{0k}}{\Delta_{k}(t_0)}e^{\lambda_{k}t}\int\limits_{-\alpha}^0  g(s) e^{\lambda_k(-\alpha- s)} ds- \frac{\varphi_{0k}(e^{-\lambda_{k}\alpha}-\lambda)}{\Delta_{k}(t_0)}\int\limits_{t}^0 g(s) e^{\lambda_k(t-s)} ds\right)v_{k}(x), \quad t<0. $$
\begin{equation}\label{um36}
 f(x)=\sum _{k=1}^{\infty }\frac{\varphi_{0k}(e^{-\lambda_{k}\alpha}-\lambda)}{\Delta_{k}(t_0)}v_{k}(x).   
\end{equation}
\end{thm}
\begin{proof}\vspace{-0.5em}
We write the series (\ref{tyech1}) as sums of two series: $I_{1}(x,t)$ and $I_{2}(x,t)$. If $I^j_{1}(x,t)$ and $I^j_{2}(x,t)$ are the corresponding partial sums, then we have:
\begin{equation}\nonumber
 -\Delta I^{j}_{1} (x,t)=\sum_{k=1}^{j}\left(\frac{\lambda_{k}\varphi_{0k}}{\Delta_{k}(t_0)}E_{\rho,1}(\lambda_{k}t^{\rho})\int\limits_{-\alpha}^0  g(s) e^{\lambda_k(-\alpha- s)} ds\right)v_{k}(x) ,  
\end{equation}

\begin{equation}\nonumber
 -\Delta I^{j}_{2} (x,t)=\sum _{k=1}^{j}\left( \frac{\lambda_{k}\varphi_{0k}(e^{-\lambda_{k}\alpha}-\lambda)}{\Delta_{k}(t_0)}\int\limits_{0}^{t}s^{\rho-1}E_{\rho,\rho}\left(-\lambda_{k}s^{\rho}\right)g(t-s)\,ds\right)v_{k}(x).
\end{equation}
Next, applying the identity $\hat{A}^{-\sigma} v_{k}(x) = \lambda_{k}^{-\sigma} v_{k}(x)$ and using Lemma \ref{lem2.1}, Lemma \ref{lem2.2} and by applying Parseval's equality, we obtain:
$$\left\| -\Delta I^{j}_{1} (x,t)\right\| _{C(\Omega )}^{2} \le C\left\| \sum_{k=1}^{j }\left(\frac{\lambda^{\sigma+1}_{k}\varphi_{0k}}{\Delta_{k}(t_0)}E_{\rho,1}(\lambda_{k}t^{\rho})\int\limits_{-\alpha}^0  g(s) e^{\lambda_k(-\alpha- s)} ds\right)v_{k}(x)   \right\| _{L_{2}(\Omega )}^{2} $$
\begin{equation}\label{N61}
 \le \frac{MC t^{-2\rho}}{\lambda_1}  \sum\limits_{k=1}^{j}|\varphi_{0k}|^{2}\lambda _{k}^{2\sigma}, \quad \tau=2\sigma>\frac{N}{2}.   
\end{equation}
By the Lemma \ref{lem2.4} and Lemma \ref{lem2.2} we have 
 $$\left\| -\Delta I^{j}_{2} (x,t)\right\| _{C(\Omega )}^{2} \le C\left\| \sum _{k=1}^{j}\lambda^{\sigma+1}_{k}\left( \frac{\varphi_{0k}(e^{-\lambda_{k}\alpha}-\lambda)}{\Delta_{k}(t_0)}\int\limits_{0}^{t}s^{\rho-1}E_{\rho,\rho}\left(-\lambda_{k}s^{\rho}\right)g(t-s)\,ds\right)v_{k}(x)  \right\| _{L_{2}(\Omega )}^{2} $$
 \begin{equation}\label{N62}
   \le\frac{MC}{\lambda_1} \left(\sum \limits_{k=1}^{j}\lambda _{k}^{2\sigma } \left|\varphi_{0k}\right|^{2}  \right) , \quad \tau=2\sigma>\frac{N}{2}.  
 \end{equation}
 
 It is easy to see that
\begin{equation}\nonumber
\left\| f(x)\right\| _{C(\Omega )}^{2}\le C\sum\limits_{k=1}^{j} \lambda^{2\sigma}\left|\varphi_{0k}\right|^{2}, 
 \quad \tau=2\sigma>\frac{N}{2}.  
\end{equation}
Therefore, if the function $\varphi_{0}(x)$ satisfies the conditions \eqref{1.28}, then the following estimates hold:
\[
\left\| -\Delta I^{j}_{1}(x,t) \right\|_{C({\Omega})}^{2} \leq C, \quad
\left\| -\Delta I^{j}_{2}(x,t) \right\|_{C(\Omega)}^{2} \leq C, \quad
\left\| f(x) \right\|_{C(\Omega)}^{2} \leq C, \quad t > 0.
\]

Thus, we conclude that $\Delta u(x,t) \in C(\overline{\Omega} \times (0,\beta])$. In particular, $u(x,t) \in C(\overline{\Omega} \times [0,\beta])$, and $f(x) \in C(\overline{\Omega})$. Theorem \ref{th7.1} is proved.
\end{proof}

\begin{thm}\label{th7.2}
Let $g(t)\in C[-\alpha,\beta]$ and $g(t)\neq 0$, $t\in [-\alpha,\beta]$ and let $\delta_k\neq 0$ for all $k$. Moreover, let the assumptions of Lemma \ref{lem7.2} or Lemma \ref{lem7.5} hold.

1) If set $\mathbb{K}_{0}$ is empty, then there exists a unique solution of the inverse problem (\ref{prob1.1})-(\ref{um8}) and it can be represented as the series in Theorem \ref{th7.1}.

2) If set $\mathbb{K}_{0}$ is not empty, then for the existence of a solution to the inverse problem (\ref{prob1.1})-(\ref{um8}), it is necessary and  sufficient that the following  conditions
$$\varphi_{0k}=(\varphi_0, v_k)=0,\,\, k\in \mathbb{K}_{0},$$
be satisfied. In this case, the solution to the inverse problem (\ref{prob1.1})-(\ref{um8}) exists, but is not unique:
\begin{equation}\label{um44}
 f(x)=\sum\limits_{k\notin\mathbb{K}_0} \frac{\delta_k\varphi_{0k}}{\Delta_{k}(t_0)}v_k(x)+\sum\limits_{k \in \mathbb{K}_{0}} f_k v_k(x),   
\end{equation}
\begin{equation}\label{um45}
 u(x,t)= \sum\limits _{k=1}^{\infty }f_{k}\left(\frac{\int\limits_{-\alpha}^0  g(s) e^{\lambda_k(-\alpha- s)} ds}{\delta_k}E_{\rho,1}(-\lambda_{k}t^{\rho})+\int\limits_{0}^{t}s^{\rho-1}E_{\rho,\rho}\left(-\lambda_{k}s^{\rho}\right)g(t-s)\,ds\right)v_{k}(x),\quad t>0,  
\end{equation}
\begin{equation}\label{um46}
u(x,t)=\sum\limits _{k=1}^{\infty }f_k\left(\frac{\int\limits_{-\alpha}^0  g(s) e^{\lambda_k(-\alpha- s)} ds}{\delta_k} e^{\lambda_k t}-\int\limits_{t}^0 g(s) e^{\lambda_k(t-s)} ds\right)v_{k}(x), \quad t<0,
\end{equation}
where if $k\notin \mathbb{K}_{0}$ then $f_k$ has the form (\ref{N16}) and if $k\in \mathbb{K}_{0}$, then $f_{k}$ are arbitrary real numbers.
\end{thm}

\begin{proof}\vspace{-0.5em}
To prove the theorem we need to show that the series (\ref{um44}), (\ref{um45}) and  (\ref{um46}) satisfy all the conditions of Definition \ref{def1}. This follows directly from the proof of Theorem \ref{th7.1}, and the proof is almost the same when any of conditions of Lemma \ref{lem7.2} or Lemma \ref{lem7.5} hold. For clarity, let us suppose that the assumptions of Lemma \ref{lem7.2} are satisfied. Series (\ref{um45}) and (\ref{um46}) are divided into two parts, following the structure given in \eqref{um44}. The second part of both these series, as stated in Lemma \ref{lem7.3}, is a finite sum of smooth functions. In the first part, the satisfaction of the series of the conditions of Definition \ref{def1} can be proved in the same way as for the series (\ref{tyech1}). Here we use the lower bound (\ref{est2}) for $\Delta_{k}(t_0)$. The convergence of the first part of \eqref{um44} is shown similarly to the series \eqref{um36}, while the second part is a sum of finitely many smooth functions.
\end{proof}

\section*{Conclusion}

In this work, a subdifusion equation with the Caputo fractional derivative of order $\rho \in (0,1)$ is studied for $t > 0$, while a classical parabolic equation is considered for $t < 0$. Following the work \cite{Dizin}, forward and inverse problems ($f(x)$ is unknown) are considered with a non-local Dezin type condition. The solutions are constructed using the classical Fourier method. The main contribution of the authors is that such non-local direct and inverse problems for mixed-type equations with fractional order have not been previously studied. In the process of studying these problems, we investigate the effect of the parameter $\lambda$ in Dezin's condition, on the existence and uniqueness of the solution. As proven above, it is shown that for certain values of $\lambda$, the uniqueness of the solution may fail, and in order to recover the solution, orthogonality conditions on the given functions $\varphi_0(x)$ and $F(x,t)$ are required.

In the future, it would be of interest to consider other types of fractional derivatives instead of the Caputo derivative, in order to investigate whether similar effects occur. Another promising direction is the study of inverse problems aimed at determining  the  fractional orders in mixed-type equations for such nonlocal problems.

\section*{Acknowledgments}

The author acknowledges financial support from the Ministry of Innovative Development
of the Republic of Uzbekistan, Grant No F-FA-2021-424.


\end{document}